\documentclass[11pt]{article}
\pdfoutput=1
\usepackage{amsmath, amsfonts, amsthm, amssymb, hyperref, setspace, geometry, mathrsfs}
\title{\textbf{On generalized covering and avoidance properties of finite groups and  saturated fusion systems}\thanks{ \scriptsize\emph{E-mail addresses:}
      zsmcau@cau.edu.cn\,(S. Zhang); zhencai688@sina.com\,(Z. Shen).}}
 
\author{Shengmin Zhang,   Zhencai Shen\\
\quad
\\
{\small College of Science,
China Agricultural University,
Beijing 100083, China}}
\date{}
\newtheorem{theorem}{Theorem}[section]

\newtheorem{lemma}[theorem]{Lemma}
\newtheorem{corollary}[theorem]{Corollary}
\theoremstyle{definition}
\newtheorem{definition}[theorem]{Definition}

\newtheorem{remark}[theorem]{Remark}
\newtheorem{example}[theorem]{Example}
\allowdisplaybreaks

\expandafter\let\expandafter\oldproof\csname\string\proof\endcsname
\let\oldendproof\endproof
\renewenvironment{proof}[1][\proofname]{%
  \oldproof[\bfseries\scshape #1]%
}{\oldendproof}

\usepackage{stackengine,scalerel}
\stackMath
\def\trianglelefteqslant{\ThisStyle{\mathrel{%
  \stackinset{r}{.75pt+.15\LMpt}{t}{.1\LMpt}{\rule{.3pt}{1.1\LMex+.2ex}}{\SavedStyle\leqslant}%
}}}
\renewcommand{\unlhd}{\trianglelefteqslant}

\renewcommand{\leq}{\leqslant}
\renewcommand{\geq}{\geqslant}
\begin{document}
\maketitle
\begin{abstract}
A subgroup $A$ of a finite group $G$ is said to be a $CAP$-subgroup of $G$, if for any chief factor $H/K$ of $G$, either $A H= AK$ or $A\cap H = A \cap K$. Let $p$ be a prime, $S$ be a $p$-group and $\mathcal{F}$ be a saturated fusion system over $S$. Then $\mathcal{F}$ is said to be supersolvable, if there exists a series of $S$, namely $1 = S_0 \leq S_1 \leq \cdots \leq S_n = S$, such that $S_{i+1}/S_i$ is cyclic, and $S_i$ is strongly $\mathcal{F}$-closed for any $i=0,1,\cdots,n$. In this paper, we first introduce the concept of strong $p$-$CAP$-subgroups, and investigate the structure of finite groups under the assumptions that some subgroups of $G$ are partial $CAP$-subgroups or strong $(p)$-$CAP$-subgroups of $G$, and obtain some criteria for a group $G$ to be $p$-supersolvable. After that, we investigate the characterizations for supersolvability of $\mathcal{F}_S (G)$ under the assumptions that some subgroups of $G$ are partial $CAP$-subgroups or strong $(p)$-$CAP$-subgroups of $G$, and obtain some criteria for a fusion system $\mathcal{F}_S (G)$ to be supersolvable. The above results improve some known results and develop some new results about $CAP$-subgroups from fusion systems. 


\end{abstract}

\section{Introduction}
All groups considered in this paper are finite, and we begin by introducing the first part of our paper, which concentrates on the theory of finite groups. We say a subgroup $A$ of a group $G$ satisfies the covering and avoidance property of $G$, if for any chief factor $H/K$ of $G$, either $HA=KA$, or $H \cap A = K\cap A$. Such a group $A$ is called a $CAP$-subgroup of $G$. The covering and avoidance property was first introduced by W. Gasch{\"u}tz in \cite{WG}, and there have been a lot of generalizations of the covering and avoidance property. These properties and characterizations of $CAP$-subgroups have been widely researched since they play an important part in the structural study of solvable and  supersolvable groups (for example \cite{EZ}). In recent years, some authors give the structure of $CAP$-subgroups ({{\cite{GS}}}),  while others give some generalizations for the original $CAP$-subgroups. In {{\cite{GG}}}, partial $CAP$-subgroups were introduced as a weaker generalization, by changing any chief factor into a chief series compared to $CAP$-subgroups.   In {{\cite{BL}}},  A. Ballester-Bolinches {\it et al.} introduced a stronger generalization named strong $CAP$-subgroup.   The reason to generalize the original $CAP$-property   is that  it does not satisfy the intermediate property (see \cite[A, Definition 10.8]{DH}).   Simutaneously,  partial $CAP$-subgroup satisfies the intermediate property as well  (see {{\cite{BL2}}}). On the other hand, these subgroups mentioned above enjoy the same properties about the inheritance of quotient group as mentioned in {{\cite[Lemma 2.6]{GG}}} and {{\cite[Lemma 3(2)]{BL}}}.

It is easy to find that only partial $CAP$-subgroup satisfies both the intermediate property and the inheritance property. Our desire is to find a new   generalization of $CAP$-subgroup which   satisfies the two properties, so that it can become a useful tool to characterize the structure of finite groups. As introduced in {{\cite{HW}}}, we have the following definition:
\begin{definition}[{\cite[Definition 1.4]{HW}}]
A subgroup $A$ of a group $G$ is called a {\it $p$-$CAP$-subgroup of $G$} if $A$ covers or avoids every $pd$-chief factor of $G$, where a  $pd$-chief factor is a  chief factor of $G$ with order divisible by $p$.
\end{definition}

It follows from {{\cite[Lemma 2.1]{GQ}}} that $p$-$CAP$-subgroup satisfies the inheritance of quotient groups. However, no evidence shows that $p$-$CAP$-subgroup satisfies the intermediate property. Therefore, we can generalize the definition by strengthening our assumption into the following one.
\begin{definition}
A subgroup $A$ of a group $G$ is called a {\it strong $p$-$CAP$-subgroup of $G$}, if for any subgroup $H$ of $G$ that contains $A$, $A$ is a $p$-$CAP$-subgroup of $H$.
\end{definition}
Now we would like to give an example to show that $p$-$CAP$-subgroups are not strong $p$-$CAP$-subgroups in general, which  shows  the concepts of   strong $p$-$CAP$-subgroups and $p$-$CAP$-subgroups are different.
\begin{example}
Consider the group $G :={\rm SL}(2,5)$. Then there exists a non-normal subgroup ${\rm SL}(2,3)$ of ${\rm SL}(2,5)$, and a subgroup $C_4$ of ${\rm SL}(2,3)$. Since there is only one non-trivial normal subgroup $C_2$ of $G$, we obtain that $C_2 / C_1$ is the only $2d$-chief factor of $G$. As $C_4$ contains $C_2$, it follows that $C_4 \cdot C_1 = C_4 = C_2 \cdot C_4$, and so $C_4$ covers $C_2 / C_1$, which implies that $C_4$ is a $2$-$CAP$-subgroup of $G$. On the other hand, consider the subgroup $C_4$ again in the group ${\rm SL}(2,3)$. One can find that $Q_8 / C_2$ is a $2d$-chief factor of ${\rm SL}(2,3)$. However, since $C_2 \leq C_4 \leq Q_8$, it yields that $Q_8 \cdot C_4 = Q_8 \neq C_4 = C_4 \cdot C_2$ and $C_4 \cap Q_8  = C_4 \neq C_2  = C_4 \cap C_2$, so that $C_4$ is not a $2$-$CAP$-subgroup of ${\rm SL}(2,3)$. Therefore $C_4$ is a  $2$-$CAP$-subgroup of $G$, but not a  strong $2$-$CAP$-subgroup of $G$.

\end{example}
The first aim of our paper is to investigate the $p$-supersolvability of finite group $G$.  Having introduced the definitions of $CAP$-subgroups, partial $CAP$-subgroups and strong $p$-$CAP$-subgroups, we get enough tools to characterize the structure of finite groups. At the beginning of Section \ref{Characterizations for $p$-supersolvability of finite groups}, we obtain that a finite group $G$ is $p$-supersolvable if certain subgroups of a Sylow $p$-subgroup of $G$ are strong $p$-$CAP$-subgroups of $G$. After that, we give a new criteria  for $p$-supersolvability of $G$ under the assumption that all maximal subgroups of a Sylow $p$-subgroup  $S$ are partial $CAP$-subgroups of $G$.
\begin{theorem}\label{characterization for G to be supersolvable partial CAP}
Let $G$ be a finite group and $S$ a Sylow $p$-subgroup of $G$, where $p$ is a prime divisor of $|G|$. If the order of $S$ is larger than $p$, every maximal subgroup  of $S$ is a partial $CAP$-subgroup of $G$, and every
cyclic subgroup of order $4$ of $G$ is a partial $CAP$-subgroup of  $G$ ($  p=2$ and $S$ is
non-abelian). Then $G$ is $p$-supersolvable.
\end{theorem}

Now we introduce the  second part of our paper, which concentrates on the fusion systems. Let $S$ be a $p$-group. In \cite{PUI06}, a category $\mathcal{F}$ named “Frobenius category” was firstly introduced by L. Puig. That is, the category $\mathcal{F}$ whose objects ${\rm Ob} (\mathcal{F})$ consist of all subgroups of $S$, and certain homomorphisms under subgroups of $S$ satisfying some axioms (see {\cite[Definition 4.1]{CR}). Its more renowned name is  fusion system}. Furthermore, $\mathcal{F}$ is said to be  saturated, if it additionally satisfies certain axioms (see {\cite[Definition 4.11]{CR}). For any finite group $G$ and a Sylow $p$-subgroup $S$ of $G$, a natural saturated fusion system namely $\mathcal{F}_S (G)$ is defined (see {\cite[Definition 1.7]{CR}}). The   notations of our second part can all be found in the book \cite{CR}, and readers can refer to this book for any symbols that we have not explicitly explained. 

The  second aim of our paper concentrates on researching a special type of saturated fusion system, namely {\it supersolvable fusion system}, which was firstly introduced by N. Su in \cite{SN}.
\begin{definition}[{\cite[Definition 1.2]{SN}}]
Let $\mathcal{F}$ be a fusion system over $S$. Then $\mathcal{F}$ is said to be {\it supersolvable}, if there exists a series 
$$1 = Q_0 \leq Q_1 \leq \cdots \leq Q_n = Q,$$
where $Q_i$ is strongly $\mathcal{F}$-closed (see {\cite[Definition 4.55 (ii)]{CR}}), such that $Q_{i+1} / Q_i$ is cyclic for all $i$.
\end{definition}
\begin{remark}
There exists a strong relationship between the fusion system $\mathcal{F}_S (G)$ and the finite group $G$, where $S$ is a Sylow $p$-subgroup of $G$. To be more precise, if $G$ is $p$-supersolvable, then any $p$-chief factor  of $G$ must has order $p$. Therefore, consider the chief series $1 \leq G_1 \leq G_2 \leq \cdots \leq G_n = G$, it follows that $G_{i+1} / G_i$ possesses a cyclic Sylow $p$-subgroup of order $p$ or 1  for all $i$. Applying {\cite[Theorem 1.4]{SN}}, we conclude that $\mathcal{F}_S (G)$ is supersolvable. On the other hand, suppose that $\mathcal{F}_S (G)$ is supersolvable. Then there exists a $p$-supersolvable group $\tilde{G}$, where $S \in {\rm Syl}_p (\tilde{G})$, such that $\mathcal{F}_S (G) = \mathcal{F}_S (\tilde{G})$ (see {\cite[Proposition 1.3]{SN}}).  However, $G$ is not $p$-supersolvable in general. For example, the simple group of Lie type $G : = {\rm PSL}(2,p^q)$ has a  cyclic Sylow $p$-subgroup $S$, where $(2,p) = 1$. Obviously, $G$ is not $p$-supersolvable. Meanwhile, since $S$ is strongly $\mathcal{F}_S (G)$-closed, we get that $S/1$ is cyclic while $S,1$ are strongly $\mathcal{F}_S (G)$-closed, so that $\mathcal{F}_S (G)$ is supersolvable.

\end{remark}
The supersolvable saturated fusion systems on finite
$p$-groups are exactly the $p$-fusion systems of supersolvable finite groups (see {\cite[Proposition 1.3]{SN}}), and more profound results have been obtained by F. Aseeri and J. Kaspczyk \cite{FJ}, for example. The second part of our paper follows the patterns of   \cite{FJ}, and aims at the criteria for   fusion systems of finite groups  $\mathcal{F}_S (G)$ to be supersolvable. As a result, we  obtain several characterizations for fusion system $\mathcal{F}_S (G)$ to be supersolvable under the assumptions that certain subgroups of $S$ are strong  $CAP$-subgroups or strong $p$-$CAP$-subgroups of $G$. In particular, we give the following   characterizations for fusion system $\mathcal{F}_S (G)$ to be supersolvable under the hypothesis that certain subgroups of $S$ are strong $CAP$-subgroups of $G$.
\begin{theorem}\label{characterization for FSG strong CAP}
Let $G$ be a finite group, $p$ be a prime divisor of $|G|$, and $S$ be a Sylow $p$-subgroup of $G$. Suppose that there exists a subgroup $D$ of $S$ with order $1 < |D| < |S|$ such that every subgroup of $S$ with order $|D|$ and every cyclic subgroup of $S$ with order  $4$ (if $S$ is non-abelian and $|D|=p=2$) is a strong $CAP$-subgroup of $G$. Then $\mathcal{F}_S (G)$  is supersolvable.
\end{theorem}
\begin{theorem}\label{main theorem}
Let $G$ be a finite group, $p$ be a prime divisor of $|G|$, and $S$ be a Sylow $p$-subgroup of $G$. Suppose that there exists a subgroup $D$ of $S$ with order $1 < |D| < |S|$ such that one of the following holds:
\begin{itemize}
\item[(1)] $p$ is odd, and every subgroup of $S$ with order $|D|$ is a strong $p$-$CAP$-subgroup of $G$;
\item[(2)] $p = 2$, every subgroup of $S$ with order $|D|$ is a strong $2$-$CAP$-subgroup of $G$, ${\rm exp} (S) \leq 2$, and  every cyclic subgroup of $S$ with order  $4$ is a strong $2$-$CAP$-subgroup of $G$ if $S$ is non-abelian and $|S| /2>|D|=p=2$.
\end{itemize} 
 Then $\mathcal{F}_S (G)$  is supersolvable.
\end{theorem}
For the case that $p$ is odd, one can see from the fact a strong $CAP$-subgroup is a strong $p$-$CAP$-subgroup that Theorem \ref{main theorem} generalizes Theorem \ref{characterization for FSG strong CAP}. However, for the case that $p=2$, it is difficult to tell which one is better.
\section{Preliminaries}
In this section, several important lemmas and theorems are introduced.  Lemma \ref{inheritance} showcases the intermediate and inheritance properties. Theorems \ref{hypercentre} and \ref{main result for characterization of hypercentre} play a  vital role in our characterizations for $G$ to be $p$-supersolvable or $\mathcal{F}_S (G)$ to be supersolvable under the assumptions that certain subgroups of $G$ are strong $p$-$CAP$-subgroups of $G$. Lemmas \ref{fusion system 1} and \ref{fusion system 2} are basic observations about supersolvability of fusion systems. 
\begin{lemma}\label{inheritance}
Let $H$ be a subgroup of $G$ and $N$ a normal subgroup of $G$. Then the following statements are true.
\begin{itemize}
\item[(1)]  $N$ is a strong $p$-$CAP$-subgroup of $G$.
\item[(2)] If $H$ is a strong $p$-$CAP$-subgroup of $G$, then   $H$ is a strong $p$-$CAP$-subgroup for any subgroup $K$ of $G$ such that  $K \geq H$.
\item[(3)] If $H$ is a $p$-$CAP$-subgroup of $G$, then  $HN/N$ is a  $p$-$CAP$-subgroup of $G/N$.
\item[(4)] If $H$ is a strong $p$-$CAP$-subgroup of $G$, then $HN/N$ is a strong $p$-$CAP$-subgroup of $G/N$. 
\end{itemize}
\end{lemma}
\begin{proof}
(1) It follows from $N \unlhd K$ for any subgroup $K$ of $G$ containing  $N$ and {{\cite[Lemma 2.1(a)]{GQ}}} that $N$ is a $p$-$CAP$-subgroup for $K$. Hence $N$ is a strong $p$-$CAP$-subgroup of $G$.

(2) Since $H$ is a strong $p$-$CAP$-subgroup of $G$, $H$ is a $p$-$CAP$-subgroup for any subgroup $R$ of $G$ containing $H$. Therefore $H$ is a $p$-$CAP$-subgroup for  any subgroup $R$ of $K$ that contains $H$. Hence we conclude that $H$ is a strong $p$-$CAP$-subgroup of $K$.

(3)  Let  $(U/N) / (V/N)$ be an arbitrary  chief factor of $G/N$ such that $|(U/N) / (V/N)|$ is a multiple of $p$. Then  $U/V$ is a  chief factor of $G$ such that $U,V \geq N$ and $|U/V|$ is a multiple of $p$. Therefore  $H$ covers or avoids this chief factor. Suppose that $H U = HV$, then we get that $HN/N \cdot U/N = H U/N = HV/N = HN/N \cdot V/N$, so that $HN/N $ covers $(U/N) / (V/N)$.  Suppose on the other hand that $H \cap V = H \cap U$. Then we obtain from the fact $NH \cap U = N (H \cap U) = N (H \cap V) = HN \cap V$ that $NH/N  \cap U/N = HN/N \cap V/N$, so that $HN/N$ avoids $(U/N) / (V/N)$. Hence  we conclude from the choice of $(U/N) / (V/N)$ that $HN/N$ is a $p$-$CAP$-subgroup of $G/N$. 

(4) Since $H$ is a strong $p$-$CAP$-subgroup of $G$, $H$ is a $p$-$CAP$-subgroup of any subgroup $K$ of $G$ that contains $HN$.  Therefore for any subgroup $K$ of $G$ such that $K$ contains $HN$, we obtain from Lemma \ref{inheritance} (3) that $HN/N$ is a $p$-$CAP$-subgroup of $K/N$. By the choice of $K$, we conclude that $HN/N$ is a strong $p$-$CAP$-subgroup of $G/N$, as required. 
\end{proof}
According to {{\cite[Section 9.5]{Ro}}}, a class of finite groups $\mathfrak{F}$ is said to be a {\it formation} if every image of an $\mathfrak{F}$-group is an $\mathfrak{F}$-group and if $G/(N_1 \cap N_2)$ belongs to $\mathfrak{F}$ whenever $G/N_1$ and $G/N_2$ belong to $\mathfrak{F}$. A formation $\mathfrak{F}$ is said to be {\it saturated} if a finite group $G \in \mathfrak{F}$ whenever $G/\Phi (G)$. 

Denote  by $\mathfrak{U}$   the class of all supersolvable groups and by $\mathfrak{U_p}$   the class of all $p$-supersolvable groups, where $p$ is a prime. Both $\mathfrak{U}$ and $\mathfrak{U_p}$ are saturated formations, and we have $\mathfrak{U_p} \supseteq \mathfrak{U}$.

Let $\mathfrak{H}$ be a non-empty class of groups. According to {{\cite[Definition 1.2.9 and Definition 2.3.18]{BB}}}, a chief factor $H/K$ of a group $G$ is  called {\it $\mathfrak{H}$-central} in $G$ if $[H/K]\ast G$ belongs to $\mathfrak{H}$, where $[H/K] \ast G$ is the semi-direct product  $[H/K](G/C_G(H/K))$ if $H/K$ is abelian and $G/C_G(H/K)$ if $H/K$ is non-abelian. A normal subgroup $N$ of a group $G$ is said to be {\it $\mathfrak{H}$-hypercentral} in $G$ if every chief factor of $G$ below $N$ is $\mathfrak{H}$-central in $G$. By the generalised Jordan-H\"older Theorem, we conclude that the product of $\mathfrak{H}$-hypercentral normal subgroups of a group $G$ is $\mathfrak{H}$-hypercentral in $G$ as well. Thus every group $G$ possesses a unique maximal normal $\mathfrak{H}$-hypercentral subgroup namely the $\mathfrak{H}$-hypercentre of $G$ which is
denoted by $Z_{\mathfrak{H}} (G)$. Applying the generalised Jordan-H\"older Theorem again,  every chief factor of $G$ below $Z_{\mathfrak{H}} (G)$ is $\mathfrak{H}$-central in $G$.

\begin{theorem}\label{hypercentre}
Let $p$ be a prime and $P$ be a normal $p$-subgroup of $G$. If every cyclic subgroup of $P$ of prime order or $4$ is a strong $p$-$CAP$-subgroup of $G$, then $P \leq Z_{\mathfrak{U}} (G)$, where $Z_{\mathfrak{U}} (G)$ denotes the $\mathfrak{U}$-hypercentre of $G$.
\end{theorem}
\begin{proof}
Assume that the theorem is not true and choose $G$ and $P$ for which it fails. Then there exists a chief factor of $G$ below $P$ which is not of prime order. Among all of the non-cyclic chief factors of $G$ below $P$, we choose $L/K$ such that $|L|$ is of minimal size.

Suppose that there exists an element $x$ of $L$ of order $p$ or $4$ which is not in $K$. Then $\langle x \rangle $ is a strong $p$-$CAP$-subgroup of $G$. Consider a chief series of $G$ passing through $L$, namely $\Gamma$:
$$1 = L_0 < L_1 < \cdots < L_{s-1} <L_s =L < L_{s+1} <\cdots < L_{m-1} <L_{m} = G. $$
It follows from the fact $L/K$ is below $P$ that $\langle x \rangle$ covers or avoids every chief factor $L_{i+1} / L_i$, $i=0,1,\cdots,s-1$. Since $L_i \geq \langle x \rangle$ for any  $i \geq s$, we conclude that $\langle x \rangle$ avoids $L_{i+1} /L_i$ for any $i \geq s$. Hence $\langle x \rangle$ covers or avoids each chief factor of $\Gamma$. By the choice of $L/K$ we get that  $L_{i+1}/L_i$  is of prime order for   $i=0,1,\cdots,s-1$. If $\langle x \rangle$ covers $L_s / L_{s-1}$, then it follows from $L_s = L_{s-1} \langle x \rangle$ that $L_s/L_{s-1}$ is of prime order. Hence we obtain that $L \leq Z_{\mathfrak{U}} (G)$. This implies that $L/K$ is of prime order, a contradiction. Hence we have that $\langle x \rangle$ avoids $L_s / L_{s-1}$. Therefore $L_s = L_s \cap \langle x \rangle = L_{s-1} \cap \langle x \rangle$, i.e. $x \in L_{s-1}$. Thus  $L_{s-1} K >K$, which yields that $L_s = L_{s-1} K$. Since $L_{s-1} K / K$ is isomorphic to $L_{s-1} / (K \cap L_{s-1})$, we conclude that $L_{s-1} / (K \cap L_{s-1})$ is a chief factor of $G$ as well. By the minimality of $|L|$, we get that $L_{s-1} / (K \cap L_{s-1})$ is of prime order, i.e. $L/K$ is of prime order, a contradiction. Hence every element of $L$ of order $p$ or $4$ is contained in $K$.

Denote by $X$   the intersection of the centralisers of the chief factors of $G$ below $K$. Since all chief factors below $K$ are of prime order, it follows that all of them are cyclic. Let 
$$1 <K_1 < K_2 < \cdots < K_q = K$$
be part of a chief series of $G$ passing through $K$. Then $[K_i,K] \leq K_{i-1}$, $i=2,3, \cdots,q$. Hence $X$ stabilises a chain of subgroups of $K$. Applying {{\cite[Chapter A, Corollary 12.4]{DH}}}, $O^p (X)$ centralises $K$. In particular,   $O^{p} (X)$ centralises every element of prime order or order $4$ of $L$. By {{\cite[Chapter IV, Satz 5.12]{H1}}}, $O^p (X)$ centralises $L$. Thus $X/C_X (L/K)$ is a normal $p$-subgroup of $G/C_X (L/K)$. By {{\cite[Chapter B, Proposition 3.12]{H1}}}, $X$ centralises $L/K$, so  that $L/K$ can be regarded as an irreducible $G/X$-module over the finite field of $p$-elements. Notice that since every chief factor $U/V$ of $G$ below $K$ is of order $p$,  we conclude that $G/C_G(U/V)$ is cyclic of order dividing $p-1$. Consequently, $G/X$ is abelian of exponent dividing $p-1$. By    {{\cite[Chapter B, Theorem 9.8]{DH}}}, $L/K$ has order $p$. This final contradiction shows that no such counterexample $G$ exists.
\end{proof}
Now we give the following characterization for $Z_{\mathfrak{U}} (G)$ under the assumption that certain subgroups of a normal $p$-subgroup $P$ of $G$  are strong $p$-$CAP$-subgroups of $G$. This characterization generalizes Theorem \ref{hypercentre} greatly  except for the case that $p=2$. 
\begin{theorem}\label{main result for characterization of hypercentre}
Let $P$ be a non-trivial normal $p$-subgroup of $G$ such that $|P| = p^n$. Suppose that $n>1$ and there exists an integer  $1<d<p^n$ with $d \,|\,p^n$ such that one of the following holds:
\begin{itemize}
\item[(1)] $p$ is odd, and every subgroup of $P$ of order $d$ is a strong $p$-$CAP$-subgroup of $G$;
\item[(2)] $p=2$,   every subgroup of $P$ of order $d$ is a strong $2$-$CAP$-subgroup of $G$, ${\rm exp} (P) \leq 2$, and every cyclic subgroup of $P$ of order $4$  is a strong $2$-$CAP$-subgroup of $G$ if $p = 2 $, $2 = d< 2^{n-1}$ and $P$ is non-abelian.
\end{itemize}
Then $P \leq Z_{\mathfrak{U}} (G)$. 
\end{theorem}
\begin{proof}
Suppose that the theorem fails, and let $(G,P)$ be a counterexample for which $|G|+|P|$ is minimal. Denote the subgroup $Z_{\mathfrak{U}} (G)$  by $Z$. 
\begin{itemize}
\item[\textbf{Step 1.}]  $d < p^{ n-1}$. 
\end{itemize}

Suppose that $d = p^{ n-1}$. We claim that, if $N$ is a minimal normal subgroup of $G$ that is contained in $P$, then $P/N \leq Z_{\mathfrak{U}} (G/N)$. Moreover, $G$ has the  unique minimal normal subgroup $N$ of $G$ contained in $P$ and $|N| > p$. Actually, it follows from Lemma \ref{inheritance} (4) that our hypothesis of the theorem holds for the pair $(G/N,P/N)$. Therefore, we obtain from the choice of $(G,P)$ that $P/N \leq Z_{\mathfrak{U}} (G/N)$. If $G$ has two different minimal normal subgroups $R$ and $N$ that are contained in $P$, then we conclude from the fact $RN/R \cong N$  that $N$ is a cyclic subgroup of order $p$. Hence $P$ is contained in $Z$, which is  impossible. Thus $N$ is the unique minimal normal subgroup of $G$ contained in $P$, and the order of $N$ must bigger than $p$, as claimed.

Now we claim that $\Phi(P) \neq 1$. Indeed, suppose that $\Phi(P) = 1$. Then $P$ must be an elementary abelian $p$-group. Let $X$ be a maximal subgroup of $N$. Notice that $|N| >p$ as proved in the preceding paragraph, we obtain that $X$ is non-trivial. Let $B$ be a complement of $N$ in $P$, and denote the product $XB$ by $H$. Then obviously $H$ is a maximal subgroup of $P$. Notice again that $d = p^{ n-1}$, $H$ is a strong $p$-$CAP$-subgroup of $G$. Therefore $H$ covers or avoids the $pd$-chief factor $N/1$, so that $NH = H$ or $H  \cap N = 1$. The first case suggests that $N \leq H$, which is  impossible. The second case implies that $X = 1$, which is also absurd. Therefore $\Phi (P) \neq 1$, as needed. 

It follows from our preceding paragraph that $\Phi(P)$ is non-trivial. Let $N$ be a minimal normal subgroup of $G$ that is contained in $\Phi (P)$. Then $(G/N,P/N)$ still satisfies the hypothesis of the theorem, so that $P/N \leq Z_{\mathfrak{U}} (G/N)$. Therefore we obtain from {\cite[Lemma 2.2 (3)]{GW2}} that $P/\Phi(P) \leq Z_{\mathfrak{U}} (G/\Phi(P))$. Applying {\cite[Lemma 2.4]{WS}}, we have that $P \leq Z$, a contradiction. Thus we conclude that  $d < p^{ n-1}$.
\begin{itemize}
\item[\textbf{Step 2.}] $d>p$. 
\end{itemize}

Suppose that $d=p$. Let $B$ be a Thompson critical subgroup of $P$ and denote the subgroup $\Omega (B)$ by $\Omega$. Let $P/S$ be a chief factor of $G$. Assume that $|S| = p$, then we get  that $S \leq Z$. Now, consider the case that $|S| >p$. Then the hypothesis holds for $(G,S)$. Therefore we obtain from the choice of $(G,P)$ that $S \leq Z$, so that $P/S$ must be non-cyclic.  Now let $V \neq P$ be any normal subgroup of $G$ which is contained in $P$. Then we have that $V \leq Z$ as the hypothesis also holds for the pair $(G,V)$. If $V \not\leq S$, then we obtain from the fact $P/S \cong VS/S \cong V / (V \cap S)$ that $P/S$ is cyclic, which is impossible. Therefore we conclude that $V \leq S$, which implies that $\Omega = P$. This is because the fact $\Omega < P$ yields that $\Omega \leq S\leq Z$, so that $P \leq Z$ by {\cite[Lemma 2.12]{GW3}}, a contradiction. Hence $\Omega = P$, as needed.

Now let $L/S$ be any minimal subgroup of $(P/S) \cap Z(M /S)$, where $M$ is a Sylow $p$-subgroup of $G$. Let $x \in L \setminus S$ and $H : = \langle x \rangle $. Then we obtain that $x$ is of order prime or 4 by the fact that $\Omega = P$ and {\cite[Lemma 2.5]{WS}}. Obviously, if $x$ is of order 4, then we must have that  $x^2 \in S$. Now, since $H$ is a strong $p$-$CAP$-subgroup by our hypothesis, we obtain that $H$ either covers or avoids $P/S$. If $H \cap P = H\cap S$,  then we obtain that $S\leq H$, which is absurd. Therefore we get that $HP = P = HS$. Notice that $|P| = |H| |S| / |H \cap S|$, if $H$ is of prime order, then $P/S$ is of prime order, a contradiction. Hence $H$ is of order 4. However, we have that $|H \cap S| \geq 2$, so that $P/S$ is still of prime order, another contradiction. Thus we conclude that  $d>p$, as desired.

\begin{itemize}
\item[\textbf{Step 3.}] $|N| \leq d$ for every minimal  normal subgroup $N$ of $G$ contained in $P$.
\end{itemize}

Assume on the other hand that $|N| > d$. Then there exists a non-trivial proper subgroup $K$ of $P$ of order $d$.  Hence   $K$ is a strong $p$-$CAP$-subgroup of $G$, so that $K$ covers or avoids  the $pd$-chief factor $N/1$. Therefore we have that $N K = K$ or $K \cap N = 1$. The first case implies that $N \leq K$, which is absurd. The second case contradicts the fact that $1 < K < N$. Thus we obtain that $|N | \leq d$, as desired. 
\begin{itemize}
\item[\textbf{Step 4.}] Suppose that  $N$ is a minimal normal subgroup of $G$ which is  contained in $P$, then the hypothesis still holds for the pair $(G/N, P/N)$.
\end{itemize}

Obviously, if either the order of $N$ is less than $d$ and $p >2$,  or the order of $N$ is less than $d/2$ and $p = 2$, then our conclusion follows directly from Lemma \ref{inheritance} (4). Suppose on the other hand that $p>2$ and $|N| = d$; or     $p=2$ and $|N| = d/2$ or $d$. Now, we claim that $|N| $ can not equal $d$ for any prime $p$. Actually, since  $N$ is not cyclic as $|N| = d \geq p^2$, we obtain that every subgroup of $P$ which contains $N$ is non-cyclic. Let $B$ be a subgroup of $P$ of order $p \cdot d$ that contains $N$. Then $B$ is not cyclic, so that $B = NT$ for some maximal subgroup $T$ of $B$. Notice that since $T$ is of order $d$, we conclude that $T$ is a strong $p$-$CAP$-subgroup of $G$, so that $T$ covers or avoids $N/1$. Therefore $N \leq T$ or $N \cap T = 1$. Notice that $N \not\leq T$, we obtain that $N \cap T =1$. As $N$ and $T$ are maximal subgroups of $B$, we conclude that $N$ is of prime order, i.e. $d = p$, a contradiction to Step 2. Hence  it suffices to deal with the case that $|N| = d/2$. 

If $|N| >2$, then $N$ must be non-cyclic. Hence we may consider the cyclic quotient $Q/N$ of order 4. Notice that $Q$ is a non-cyclic subgroup of $P$, we get that there exists two maximal subgroups $A$ and $B$ of $Q$ such that $Q = AB$ and $N \leq B$. If $N \not\leq B$, then we obtain that $Q = NB$, so that $Q/N = BN/N$ is a strong $2$-$CAP$-subgroup of $G/N$ by Lemma \ref{inheritance} (4).  Suppose that $N \leq B$, then we conclude  that $A/N$, $B/N$ are two different subgroups of $Q/N$ of order 2, so that $Q/N$ is not cyclic, which is impossible. Now it suffices to deal with the case that $|N| = 2$ and $d =4$.

It is easy to see from Lemma \ref{inheritance} (4) that all subgroups of $P/ N$ of order 2 are strong $2$-$CAP$-subgroups of $G/N$. Let $Q/N$ be an arbitrary cyclic subgroup of $P/N$ such  that $Q/N$ is of order 4. Suppose that $Q$ is non-cyclic, then there exists two maximal subgroups $X,Y$ of $Q$. If $X \not\geq N$, then $Q/N = XN/N$, so that $Q/N$ is  a strong $2$-$CAP$-subgroup of $G/N$ by Lemma \ref{inheritance} (4). A similar argument yields the same for $Y$, and so we only need to consider the situation that $N \leq X,Y$. However, this implies that $Q/N$ is non-cyclic, a contradiction. Hence $Q$ is cyclic. Notice that $Q$ is of order 8, this indicates that ${\rm exp} (P) \geq 3$, a contradiction as well. Thus we finally conclude that all cyclic  subgroups of $P/N$ of order 4 are strong $2$-$CAP$-subgroups of $G/N$, so that Step 4 is finished.
\begin{itemize}
\item[\textbf{Step 5.}] If $ P/N\leq  Z_{\mathfrak{U}}(G/N)$ holds for every minimal normal subgroup $N$ of $G$ which is  contained in $P$, then $\Phi(P) \neq  1$.
\end{itemize}

Suppose that $\Phi(P) = 1$. Then we obtain that $P$ is an elementary abelian $p$-group. Assume that  there  exists a minimal normal subgroup $S \neq N$ of $G$ which is contained in $P$. Since we have $P/S \leq Z_{\mathfrak{U}} (G/S)$ by Step 4, and we obtain from the fact $N \cong NS/S$ that $N \leq Z$, we conclude that $P\leq Z$, which is absurd. Therefore $N$ is the unique minimal normal subgroup of $G$ which is contained in $P$ and $N \not\leq Z$. Thus $Z \cap P = 1$ and $|N| >p$. 

Let $X$ be a maximal subgroup of $N$. Denote by $D$ a complement of $N$ in $P$ and by $B$ a subgroup of $D$ such that $XB$ is of order $d$. Then we obtain that $XB$ is a strong $p$-$CAP$-subgroup of $G$. However, this implies that $XB$ either covers or avoids the chief factor $N/1$ of $G$, i.e. $XB \leq N$ or $XB \cap N = 1$. The first case induces that $B \leq N$, impossible. The second case suggests that $X \cap N = 1$, which is also absurd. Hence we   obtain that $\Phi(P) \neq 1$. 
\begin{itemize}
\item[\textbf{Step 6.}] Final contradiction.
\end{itemize}

It follows from Step 4 that the pair $(G/N, P / N)$  satisfies the hypothesis of the theorem for every minimal normal subgroup $N$ of $G$   contained in $P$, and so $ P/N\leq  Z_{\mathfrak{U}}(G/N)$ holds for every minimal normal subgroup $N$ of $G$ which is  contained in $P$. Applying Step 5 we conclude that $\Phi(P) \neq 1$. Applying {\cite[Lemma 2.2 (3)]{GW2}} again, we have  that $P/\Phi(P) \leq Z_{\mathfrak{U}} (G/\Phi(P))$. Therefore, we obtain from {\cite[Lemma 2.4]{WS}} that $P \leq Z$. The final contradiction proves the theorem.
\end{proof}
Since every cyclic chief factor of order $2$ in a chief series of a group is central, applying Theorem \ref{hypercentre} and Theorem \ref{main result for characterization of hypercentre}, we have the following corollary.
\begin{corollary}
Let $P$ be a normal $2$-subgroup of $G$. Suppose that:
\begin{itemize}
\item[(1)] every cyclic subgroup of $P$ of order $2$ or $4$ is a strong $2$-$CAP$-subgroup of $G$;
\item[(2)] ${\rm exp} (P) \leq 2$ and  there exists an integer $2 <d < |P|$ with $d\,|\,|P|$ such that every subgroup of $P$ of order $d$ and every   cyclic subgroup of $P$ of order    $4$ (if $d<|P| /2$ and $P$ is non-abelian) is a strong $2$-$CAP$-subgroup of $G$.
\end{itemize}
Then $P$ lies in the nilpotent hypercentre of $G$, namely $Z_{\infty} (G)$.
\end{corollary}

In the 2024 paper of F. Aseeri and J. Kaspczyk \cite{FJ}, useful techniques were provided to give new criteria of supersolvable fusion systems.
\begin{lemma}[{{\cite[Lemma 2.6]{FJ}}}]\label{fusion system 1}
Let $p$ be a prime and $\mathcal{F}$ be a saturated fusion system on a finite $p$-group $S$. Assume that the fusion system $N_{\mathcal{F}} (Q)$ is supersolvable for any $Q \in \mathcal{E}_{\mathcal{F}} ^{*}$, then $\mathcal{F}$ is supersolvable.\end{lemma}

\begin{lemma}[{{\cite[Lemma 2.9]{FJ}}}]\label{fusion system 2}
Let $G$ be a finite group, $p \in \pi(G)$, and $S$ be a Sylow $p$-subgroup of $G$. Suppose that for any proper subgroup $H$ of $G$ with $O_p (G) <S \cap H$ and $S \cap H \in {\rm Syl}_p (H)$, the fusion system $\mathcal{F}_{S \cap H} (H)$ is supersolvable. Assume additionally that $O_p (G) \leq Z_{\mathfrak{U}} (G)$. Then $\mathcal{F}_{S} (G)$ is supersolvable.
\end{lemma}
\begin{lemma}[{{\cite[Proposition 3.1]{WS}}}]\label{lemma GWB}
Let $P$ be a non-identity normal subgroup of $G$ with $|P|=p^n$, where $p$ is a prime divisor of $|G|$. Suppose that $n>1$, and there exists an integer $n>k \geq 1$ that any subgroup of $P$ of order $p^k$ and any cyclic subgroup of $P$ of order $4$ (if $P$ is a non-abelian 2-group and $n-1>k=1$) are $\mathfrak{U}$-embedded in $G$. Then $P \leq Z_{\mathfrak{U}} (G)$.
\end{lemma}
\section{Characterizations for $p$-supersolvability of finite groups}\label{Characterizations for $p$-supersolvability of finite groups}
\begin{theorem}\label{characterization for G minimal order}
Let $p$ be a prime and $G$ be a group of finite order. Suppose that every cyclic subgroup of $G$ with order $p$ or $4$ (if $p=2$) is a strong $p$-$CAP$-subgroup of $G$. Then $G$ is $p$-supersolvable.
\end{theorem}
\begin{proof}
Suppose that the theorem is false and let $G$ be a counterexample of minimal order. We break the argument into separately stated steps.
\begin{itemize}
\item[\textbf{Step 1.}] $G$ is a minimal non-$p$-supersolvable group.
\end{itemize}

Let $M$ be a maximal subgroup of $G$, and $L$ be a subgroup of $M$ of order $p$ or $4\,(p=2)$. Applying Lemma \ref{inheritance} (2), we conclude that $L$ is a strong $p$-$CAP$-subgroup of $M$, and so  $M$ satisfies the hypothesis of the theorem. By the minimality of $G$, $M$ is $p$-supersolvable. As a result, it follows from the choice of $M$ that every proper subgroup of $G$ is $p$-supersolvable, i.e. $G$ is a minimal non-$p$-supersolvable group.
\begin{itemize}
\item[\textbf{Step 2.}] $O_{p^{'}} (G)=1$, and $F (G) = O_p (G)$.
\end{itemize}

Set $\overline{G} = G/O_{p'} (G)$. Assume that $\overline{L}$ is a cyclic subgroup of $\overline{G}$ of order $p$ or $4\,(p=2)$. Then we can write $\overline{L} = L O_{p'} (G) /O_{p'} (G)$, where $L$ is a cyclic subgroup of $G$ of order $p$ or $4$. It follows from our assumption that $L$ is a strong $p$-$CAP$-subgroup of $G$. Therefore we conclude from Lemma \ref{inheritance} (4) that $L O_{p'} (G)/O_{p'} (G) = \overline{L}$ is a strong $p$-$CAP$-subgroup of $\overline{G}$. Therefore $\overline{G}$ satisfies the hypothesis of the theorem. Suppose that $O_{p'} (G)>1$, then $\overline{G}$ is $p$-supersolvable by our choice of $G$. By the inheritance of $p$-supersolvability we get that $G$ is $p$-supersolvable, which contradicts the fact that $G$ is not $p$-supersolvable and Step 2 is finished.
\begin{itemize}
\item[\textbf{Step 3.}] $F(G) \leq Z_{\mathfrak{U}} (G)$.
\end{itemize}

Since $F(G) =O_p (G)$ is a normal $p$-subgroup of $G$,  it follows from Theorem \ref{hypercentre} that $O_p (G) \leq Z_{\mathfrak{U}} (G)$.
\begin{itemize}
\item[\textbf{Step 4.}] $G = F^{*} (G)$.
\end{itemize}

Suppose that $F^{*} (G) < G$. Then by the $p$-supersolvability of $F^{*} (G)$ we get that $F^{*} (G) = F(G)=O_p (G)$. Applying  Step 3, we conclude that $F(G)$ is $\mathfrak{U}$-hypercentral in $G$. It follows from {{\cite[Chapter IV, Theorem 6.10]{DH}}} that $[G^{\mathfrak{U}},F(G)]=1$, i.e. $G / C_G (F(G)) \in \mathfrak{U}$. By the fact that $C_G (F^{*} (G)) \leq F^{*} (G)$ and the inheritance of supersolvability, we have  that $G/F(G) \in \mathfrak{U}$. Since every $G$-chief factor below $F(G)$ is of order $p$, it follows that $G$ is $p$-supersolvable, a contradiction. Therefore we obtain that $G =F^{*} (G)$.
\begin{itemize}
\item[\textbf{Step 5.}] $G/Z(G)$ is a non-abelian simple group  and $G$ is perfect. 
\end{itemize}

We obtain from Step 4 that $G = F^{*} (G) = F(G) E(G)$. Since $G$ is not nilpotent, it follows that $E(G)$ is not contained in $F(G)$. Hence we can find a component $H$ of $G$. Then $H$ is normal in $F^{*} (G) = G$, and $H/Z(H)$ is a non-abelian simple group. Applying Step 2,  $|H|$ is divisible by $p$. In particular, $H$ is not $p$-supersolvable, and so $H=G$, i.e. $G/Z(G)$ is a  non-abelian simple group and $G$ is perfect. 
\begin{itemize}
\item[\textbf{Step 6.}] Final contradiction.
\end{itemize}

Let $M$ be a maximal subgroup of $G$ which does not contain $Z(G)$. Then $M Z(G) =G$ and   $M$ is normal in $G$. Since $G/M =Z(G) M/M \cong Z(G) / M \cap Z(G)$, it follows that $G^{'} \leq M$, contradicting the fact that $G$ is perfect. Thus no such $M$ exists and $Z(G) \leq \Phi (G)$. Now let $G/A$ be a chief factor of $G$. Since $A$ is normal in $G$, we get  that $A \leq Z(G)$, so that $A =Z(G)$. Hence $Z(G)$  lies in all chief series of $G$. If $p \,|\,|G/Z(G)|$, let $x$ be an arbitrary element of $G$ of order $p$ or $4\,(p=2)$. It follows from our hypothesis that $\langle x \rangle$ is a strong $p$-$CAP$-subgroup of $G$. Therefore $\langle x \rangle$ covers or avoids $G/Z(G)$. If $G = G \langle x \rangle = Z(G) \langle x \rangle$, it follows that $G/Z(G)$ is abelian, a contradiction to the fact that $G/Z(G)$ is non-abelian by Step 5. Therefore  we have that $G \cap \langle x \rangle = \langle x \rangle = Z(G) \cap \langle x \rangle$, i.e. $x \in Z(G)$. By the choice  of $x$, we conclude from {{\cite[Chapter IV, Satz 5.5]{H1}}} that $G$ is $p$-nilpotent, contradicting the fact that $G$ is not $p$-supersolvable. Thus  $ |G/Z(G)|$ is not divisible by $p$. If $Z(G)$ has non-trivial $p'$-part, it follows that $O_{p'} (G) >1$, a contradiction to Step 2. Thus $Z(G)$ is a $p$-group and  so $Z(G) \leq O_p (G) = F(G)$. However, since $G/ Z(G)$ is a non-abelian simple group, it follows from $F(G) \unlhd G$ that $F(G) \leq Z(G)$.
Thus we get that $F(G) =Z(G)$. By Step 2, we conclude that every $G$-chief factor below $Z(G)$ is of order $p$, so  that $G$ is $p$-supersolvable, a contradiction. Finally, we obtain  that no such counterexample of $G$ exists and we are done.
\end{proof}
\begin{theorem}
Let $\mathfrak{F}$ be a saturated formation containing $\mathfrak{U}$, and let $G$ be a finite group with a normal subgroup $H$ such that $G/H \in \mathfrak{F}$. If, for any $p \in \pi(G)$, every cyclic subgroup of $F^{*} (H)$ of order $p$ or $4$ (if $p=2$) is a strong $p$-$CAP$-subgroup of $G$, then $G \in \mathfrak{F}$. 
\end{theorem}
\begin{proof}
Assume that $H >1$, and let $p$ be a prime such that $p \,|\,|F^{*} (H)|$. By the hypothesis of our theorem, Lemma \ref{inheritance}(2) and Theorem \ref{characterization for G minimal order}, we conclude that $F^{*} (H)$ is $p$-supersolvable for all primes $p \,|\,|F^{*} (H)|$. Therefore $F^{*} (H)$ is supersolvable and so $F^{*} (H) = F(H)$. By  the definition of $F(H)$ and Theorem \ref{hypercentre}, we obtain that $F^{*} (H) = F(H) \leq Z_{\mathfrak{U}} (G)$. Applying  {{\cite[Chapter IV, Theorem 6.10]{DH}}}, we get  that $G/C_G (F(H)) \in \mathfrak{U}$. Since $G/H \in \mathfrak{F}$ and $\mathfrak{F} \supseteq \mathfrak{U}$, it yields that $G/C_H (F(H)) \in \mathfrak{F}$. As $C_H (F^{*} (H)) \leq F^{*} (H)$, we see from the inheritance property of $\mathfrak{F}$ that $G/F^{*} (H) = G/F(H) \in \mathfrak{F}$. Since each $G$-chief factor below $F(H) \leq Z_{\mathfrak{U}} (G)$ is of prime order,   we get that $G \in \mathfrak{F}$, as needed.
\end{proof}

\begin{proof}[Proof of Theorem \ref{characterization for G to be supersolvable partial CAP}]
Set $d =|S| \cdot p^{-1}$. If $d = 1$, then we obtain that $|S| = p$, a contradiction to the fact that $|S| > p$. Hence we must have that $d >1$.  Then we obtain that every subgroup of $S$ of order $d$ is a partial $CAP$-subgroup of $G$, every cyclic subgroup of order 4 of $G$ is a partial $CAP$-subgroup of  $G$ (if $  p=2$ and $S$ is
non-abelian). Applying \cite[Theorem C]{BL2}, we conclude that $G$ is a $p$-solvable group. Now, we claim that $G$ is $p$-supersolvable.

Suppose that our claim is false, in other words, there exists a group $G$ satisfying such conditions but $G$ is not $p$-supersolvable. Without loss of generality, let $G$ be a counterexample with $|G|$ minimal. We would like $M$ to be a minimal normal subgroup of $G$. Since $G$ is already $p$-solvable, it follows that $M$ is either a $p$-group or a $p'$-group.  If $M$ is a $p'$-group, let $Q/M$ be a maximal subgroup of $SM/M$. Then $Q \leq SM$, and so $M = O_{p'} (Q)$. Using the  Schur-Zassenhaus Theorem, we get that there exists a $p$-group $Q_1$ such that $Q_1 \ltimes M = Q$. Denote by $S_1$ the Sylow $p$-subgroup of $SM$ containing $Q_1$. Then $S_1 ^g = S$, where $g $ is an element of $SM$ by Sylow Theorem. Therefore $(Q_1 M)^g = Q_1 ^g M = Q \unlhd SM$, and so $Q/M = Q_1 ^g M/M$ can be considered as the image of a maximal subgroup $Q_1 ^g $ of $S$ under the quotient morphism induced by $M$.  Applying \cite[Lemma 2]{FY}, we obtain that every maximal subgroup $Q/M$ of $M S / M$ is a partial $CAP$-subgroup of $G/M$. Now suppose that $M$ is a $p$-group. Then it follows that $M \leq O_p (G) \leq S$. Consider the maximal subgroup $Q/M$ of $S/M$. Then we see that $Q$ is a maximal subgroup of $S$ containing $M$, and we conclude again from \cite[Lemma 2]{FY} that $Q/M$ is a partial $CAP$-subgroup of $G/M$. Therefore, every maximal subgroup of  $G/M$ is a partial $CAP$-subgroup, and so we obtain from the minimal choice of $G$ that $G/M$ is $p$-supersolvable.

Notice that since the  class of $p$-supersolvable groups is a saturated formation $\mathfrak{U}_p$, we may assume that $M$ is the unique minimal normal subgroup of $G$. Assume that $M$ is  a $p'$-group. Then it follows from the fact  $G/M$ is $p$-supersolvable that $G$ is $p$-supersolvable, a contradiction. Now, suppose on the contrary that $M$ is a $p$-group. Then we claim that $M$ is not contained in $\Phi (S)$. If $M \leq \Phi (S)$, then we obtain from the fact $\mathfrak{U}_p$ is a saturated formation that $G$ is $p$-supersolvable, another contradiction. Therefore $M \not\neq \Phi (S)$, and so there exists   a maximal subgroup $S_1$ of $S$ such that $M \not\leq S_1$. Hence $S_1$ does not covers the chief factor $M/1$. Since $S_1$ is a partial $CAP$-subgroup of $G$, we obtain from the uniqueness of $M$ that $S_1$ must avoid the chief factor $M / 1$. Therefore $M \cap S_1 = 1 \cap S_1 = 1$, and so $M$ is a cyclic subgroup of order $p$ since $M \cdot S_1 = S$. Thus, as $G/M$ is $p$-supersolvable, it follows directly that $G$ is $p$-supersolvable, which leads to the final contradiction. Thus  our claim is true, and so $G$ is  $p$-supersolvable. 
\end{proof}
As  direct applications, we obtain the following two corollaries of characterizations for the $p$-supersolvability of $G$ under the assumptions that certain subgroups of $S \in {\rm Syl}_p (G)$ are $CAP$-subgroups or strong $p$-$CAP$-subgroups of $G$.
\begin{corollary}\label{corollary CAP}
Let $G$ be a finite group and $S$ a Sylow $p$-subgroup of $G$, where $p$ is a prime divisor of $|G|$. If the order of $S$ is larger than $p$, every maximal subgroup  of $S$ is a  $CAP$-subgroup of $G$, and every
cyclic subgroup of order $4$ of $G$ is a  $CAP$-subgroup of  $G$ ($  p=2$ and $S$ is
non-abelian), then $G$ is $p$-supersolvable.
\end{corollary}
\begin{corollary}
Let $G$ be a finite group and $S$ a Sylow $p$-subgroup of $G$, where $p$ is a prime divisor of $|G|$.  Suppose that the order of $S$ is larger than $p$, every maximal subgroup  of $S$ and every
cyclic subgroup of order $4$ of $G$ ($  p=2$ and $S$ is
non-abelian) is a strong $q$-$CAP$-subgroup of $G$ for any prime divisor $q$ of $|G|$. Then $G$ is $p$-supersolvable.
\end{corollary}
\begin{proof}
Let $H$ be a subgroup of $G$ such that $H$ is a strong $q$-$CAP$-subgroup of $G$ for any prime divisor $q$ of $|G|$. Then for any chief factor $U/V$ of $G$, it follows that there exists a prime divisor $r$ of $|G|$ such that $r \,|\,|U/V|$. Therefore we conclude that $H$ covers or avoids $U/V$, and we  obtain from the choice of $U/V$ that $H$ is a $CAP$-subgroup of $G$. Hence every maximal subgroup  of $S$ and every
cyclic subgroup of  $G$ of order 4  (if $  p=2$ and $S$ is
non-abelian) are   $CAP$-subgroups of $G$,   so that $G$ is $p$-supersolvable by Corollary \ref{corollary CAP}, as needed.
\end{proof}

\section{Characterizations for  supersolvability of saturated  fusion systems}\label{Characterizations for  supersolvability of saturated  fusion systems}
In this section, we investigate the structure of $\mathcal{F}_S (G)$ under the assumptions that certain subgroups of $G$ are strong $p$-$CAP$-subgroups, strong $CAP$-subgroups or partial $CAP$-subgroups of $G$, and obtain several criteria for fusion system $\mathcal{F}_S (G)$ to be supersolvable and finite group $G$ to be $p$-nilpotent. Methods used in this section are mainly related to fusion systems, which are totally different from those of Section \ref{Characterizations for $p$-supersolvability of finite groups}. At the beginning, we prove a result a little weaker compared to Theorem \ref{characterization for G minimal order} using the method of fusion systems, to showcase the universality of fusion system methods.

\begin{theorem}\label{characterization for FSG strong p-CAP}
Let $G$ be a finite group, $p$ be a prime divisor of $|G|$, and $S$ be a Sylow $p$-subgroup of $G$. Suppose that any cyclic subgroup of $S$ with order $p$ or $4$ (if $p$=2) is a strong $p$-$CAP$-subgroup of $G$.  Then $\mathcal{F}_S (G)$  is supersolvable.
\end{theorem}
\begin{proof} 
Assume that the theorem is false, and let $G$ be a counterexample of  minimal order. Now denote $\mathcal{F}_S (G)$ by $\mathcal{F}$.
\begin{itemize}
\item[\textbf{Step 1.}] Let $H$ be a proper subgroup of $G$ such that $S \cap H \in {\rm Syl}_p (H)$ and $|S \cap H| \geq p^2$. Then $\mathcal{F}_{S \cap H} (H)$ is supersolvable.
\end{itemize}

By our hypothesis, every cyclic subgroup $T$ of $S \cap H$ with order $p$ or $4$ (if $p=2$) is a strong $p$-$CAP$-subgroup of $G$. Then every cyclic subgroup $T$ of $S \cap H$ with order $p$ or $4$ (if $p =2$) is a strong $p$-$CAP$-subgroup of $H$ by Lemma \ref{fusion system 1} (2). Hence  $H$ satisfies the hypothesis of the theorem and it follows from the minimal choice of $G$ that $\mathcal{F}_{S \cap H} (H)$ is supersolvable.
\begin{itemize}
\item[\textbf{Step 2.}] Let $Q \in \mathcal{E}_{\mathcal{F}} ^{*}$, then $|Q| \geq p^2$. If moreover that $Q \not\unlhd G$, then $N_{\mathcal{F}} (Q)$ is supersolvable.
\end{itemize}

Suppose that there exists a subgroup $Q \in \mathcal{E}_{\mathcal{F}} ^{*}$ such that $|Q| < p^2$. Then there is a subgroup $R$ of $S$ such that $|R| = p $, and $Q <R$. It follows directly that $R \leq C_S (Q)$. Since $Q < R \leq S$, we conclude from the fact $Q$ is a member of $\mathcal{E}_{\mathcal{F}} ^{*}$ that $Q$ is $\mathcal{F}$-essential. Hence it follows from the definition of $\mathcal{E}_{\mathcal{F}} ^{*}$ that $Q$ is $\mathcal{F}$-centric. Therefore $R \leq C_S (Q) =Z(Q) \leq Q$, a contradiction. Thus we have that $|Q| \geq p^2$.

Assume that $Q$ is not normal in $G$. Then $N_G (Q)$ is a proper subgroup of $G$. Since $Q \in \mathcal{E}_{\mathcal{F}} ^{*}$, $Q$ is fully $\mathcal{F}$-normalized or $Q=S$. Clearly $S$ is fully $\mathcal{F}$-normalized, hence $Q$ is always fully $\mathcal{F}$-normalized. By the argument below {{\cite[Definition 2.4]{AK}}}, $S \cap N_G (Q) = N_S (Q) \in {\rm Syl}_p (N_G(Q))$. Since $|N_S (Q)| \geq |Q| \geq p^2$, it  follows that $N_G (Q)$ satisfies the hypothesis of Step 1, and so $\mathcal{F}_{N_S (Q)} (N_G(Q))=N_{\mathcal{F}} (Q)$ is supersolvable.
\begin{itemize}
\item[\textbf{Step 3.}] $|O_p (G)| \geq p^2$.
\end{itemize}

Assume that there does not exist a subgroup $Q \in \mathcal{E}_{\mathcal{F}} ^{*}$ such that $Q \unlhd G$. Then for each $Q \in  \mathcal{E}_{\mathcal{F}} ^{*}$, the fusion system $N_{\mathcal{F}} (Q)$ is supersolvable by Step 2. By Lemma \ref{fusion system 1}, $\mathcal{F}$ is supersolvable, a contradiction. Thus there exists a subgroup $Q \in \mathcal{E}_{\mathcal{F}} ^{*}$ such that $Q \unlhd G$. Hence we conclude from Step 2 that $|O_p (G)| \geq |Q| \geq p^2$.
\begin{itemize}
\item[\textbf{Step 4.}] $O_p (G) \leq Z_{\mathfrak{U}} (G)$.
\end{itemize}

It follows from the inequality $|O_p (G)| \geq p^2$ that every subgroup $T$ of $O_p (G)$ of order $p$ and every cyclic subgroup $T$ of $O_p (G)$ of order $4$ (if  $p =2$) is a strong $p$-$CAP$-subgroup of $G$. Thus we obtain from Theorem \ref{hypercentre} that $O_p (G) \leq Z_{\mathfrak{U}} (G)$ and this part is complete.
\begin{itemize}
\item[\textbf{Step 5.}] Final contradiction.
\end{itemize}

Suppose that $H$ is a proper subgroup of $G$ such that $O_p (G) < S \cap H$ and $S \cap H \in {\rm Syl}_p (H)$. By Step 1 and Step 3, $|S \cap H| >|O_p (G)| \geq p^2$ and so $\mathcal{F}_{S \cap H} (H)$ is supersolvable. Since $O_p (G) \leq Z_{\mathfrak{U}} (G)$ by Step 4, it follows directly from Lemma \ref{fusion system 2} that $\mathcal{F}_S (G)$ is supersolvable, a contradiction. Hence our proof is complete.
\end{proof}
As a direct application of the theorem above, we obtain the following characterization for the structure of finite groups under the assumption that every cyclic subgroup of $S$ with order $p$ or $4$ (if $p$=2) is a strong $p$-$CAP$-subgroup of $G$. 
\begin{corollary}
Let $G$ be a finite group and $S$ a Sylow $p$-subgroup of $G$, where $p$ is a  prime divisor of $|G|$ such that $(p-1, |G|)=1$. Suppose that any cyclic subgroup of $S$ with order $p$ or $4$ (if $p$=2) is a strong $p$-$CAP$-subgroup of $G$. Then $G$ is $p$-nilpotent.
\end{corollary}
\begin{proof}
It follows from Theorem \ref{characterization for FSG strong p-CAP} that $\mathcal{F}_S (G)$  is supersolvable. Thus we conclude from {{\cite[Theorem 1.9]{SN}}} that $G$ is $p$-nilpotent, as desired.
\end{proof}
\begin{proof}[Proof of Theorem \ref{characterization for FSG strong CAP}]
Assume that the theorem is false, and let $G$ be a counterexample of  minimal order. Now denote $\mathcal{F}_S (G)$ by $\mathcal{F}$.
\begin{itemize}
\item[\textbf{Step 1.}] Suppose that $H$ is a proper subgroup of $G$ such that $S \cap H \in {\rm Syl}_p (H)$ and $|S \cap H| \geq p|D|$. Then  the fusion system $\mathcal{F}_{S \cap H} (H)$ is supersolvable.
\end{itemize}

By our hypothesis, every  subgroup $T$ of $S \cap H$ with order $|D|$   or $4$ (if $S$ is non-abelian and $|D| = p = 2$) is a strong $CAP$-subgroup of $G$. Then every cyclic subgroup $T$ of $S \cap H$ with order $|D|$ or $4$ (if $S \cap H$ is non-abelian and $|D| = p = 2$) is a strong $CAP$-subgroup of $H $ by the definition of strong $CAP$-subgroup {\cite{BL}}. Hence  $H$ satisfies the hypothesis of the theorem and it follows from the minimal choice of $G$ that $\mathcal{F}_{S \cap H} (H)$ is supersolvable.
\begin{itemize}
\item[\textbf{Step 2.}] Let $Q \in \mathcal{E}_{\mathcal{F}} ^{*}$, then $|Q| \geq p|D|$. If moreover that $Q \not\unlhd G$, then $N_{\mathcal{F}} (Q)$ is supersolvable.
\end{itemize}

Suppose that there exists a subgroup $Q \in \mathcal{E}_{\mathcal{F}} ^{*}$ such that $|Q| < p|D|$. Then there is a subgroup $R$ of $S$ such that $|R| = |D| $, and $Q <R$. It follows directly that $R \leq C_S (Q)$. Since $Q < R \leq S$, we conclude from the fact $Q$ is a member of $\mathcal{E}_{\mathcal{F}} ^{*}$ that $Q$ is $\mathcal{F}$-essential. Therefore $Q$ is $\mathcal{F}$-centric. Hence $R \leq C_S (Q) =Z(Q) \leq Q$, a contradiction. Thus we get that $|Q| \geq p|D|$. Now assume that $Q$ is not normal in $G$. Therefore $N_G (Q)$ is a proper subgroup of $G$. Since $Q \in \mathcal{E}_{\mathcal{F}} ^{*}$, $Q$ is fully $\mathcal{F}$-normalized or $Q=S$. Clearly $S$ is fully $\mathcal{F}$-normalized, hence $Q$ is always fully $\mathcal{F}$-normalized. By the argument below {{\cite[Definition 2.4]{AK}}}, $S \cap N_G (Q) = N_S (Q) \in {\rm Syl}_p (N_G(Q))$. Since $|N_S (Q)| \geq |Q| \geq p|D|$, it follows that $N_G (Q)$ satisfies the hypothesis of Step 1, and so $\mathcal{F}_{N_S (Q)} (N_G(Q))=N_{\mathcal{F}} (Q)$ is supersolvable.
\begin{itemize}
\item[\textbf{Step 3.}] $|O_p (G)| \geq p|D|$, and moreover $O_p (G) \leq Z_{\mathfrak{U}} (G)$. 
\end{itemize}

Assume that there does not exist a subgroup $Q \in \mathcal{E}_{\mathcal{F}} ^{*}$ such that $Q \unlhd G$. Then for each $Q \in  \mathcal{E}_{\mathcal{F}} ^{*}$, the fusion system $N_{\mathcal{F}} (Q)$ is supersolvable by Step 2. By Lemma \ref{fusion system 1}, $\mathcal{F}$ is supersolvable, a contradiction. Thus there exists a subgroup $Q \in \mathcal{E}_{\mathcal{F}} ^{*}$ such that $Q \unlhd G$. Hence we conclude from Step 2 that $|O_p (G)| \geq |Q| \geq p|D|$. Now we claim that $O_p (G) \leq Z_{\mathfrak{U}} (G)$. Actually, it follows from $|O_p (G)| \geq p|D|$ that any subgroup $T$ of $O_p (G)$ of order $|D|$ or any cyclic subgroup $T$ of $O_p (G)$ of order $4$ (if  $S$ is non-abelian, $|D| = p =2$) is a strong $CAP$-subgroup of $G$. Now let $W$ be an arbitrary strong $CAP$-subgroup of $G$. Then for any chief factor $U/V$ of $G$, $W$ avoids or covers $U/V$. If $W$ covers $U/V$, i.e. $WV \geq U$, then we conclude from the Dedekind modular law that $V ( U \cap W) = W V \cap U = U$, and so $|U : V ( U \cap W)| = |G : N_G (V ( U \cap W))|=1$, which implies that $W$ is a generalized $CAP$-subgroup by {\cite[Definition 1.1]{WS}}. Notice that $WG = G$ and  $G \cap W =W$, where $G$ is obviously a $K$-$\mathfrak{U}$-subnormal subgroup of $G$ by {\cite[Definition 6.1.4]{BB}}, we obtain from {\cite[Definition 1.2]{WS}} that $W$ is $\mathfrak{U}$-embedded in $G$. By the choice of $W$, any subgroup $T$ of $O_p (G)$ of order $|D|$ or any cyclic subgroup $T$ of $O_p (G)$ of order $4$ (if  $S$ is non-abelian, $|D| = p =2$) is $\mathfrak{U}$-embedded in $G$. Therefore, applying Lemma \ref{lemma GWB}, it follows directly that $O_p (G) \leq Z_{\mathfrak{U}} (G)$ and this part is complete.
\begin{itemize}
\item[\textbf{Step 4.}] Final contradiction.
\end{itemize}

Suppose that $H$ is a proper subgroup of $G$ such that $O_p (G) < S \cap H$ and $S \cap H \in {\rm Syl}_p (H)$. By Step 1 and Step 3, $|S \cap H| >|O_p (G)| \geq p|D|$, and so $\mathcal{F}_{S \cap H} (H)$ is supersolvable. Since $O_p (G) \leq Z_{\mathfrak{U}} (G)$ by Step 3, it follows directly from Lemma \ref{fusion system 2} that $\mathcal{F}_S (G)$ is supersolvable, a contradiction. Hence our proof is complete.
\end{proof}

\begin{corollary}\label{corollary FSG strong q-CAP}
Let $G$ be a finite group, $p$ be a prime divisor of $|G|$, and $S$ be a Sylow $p$-subgroup of $G$. Suppose that there exists a subgroup $D$ of $S$ with order $1 < |D| < |S|$ such that every subgroup of $S$ with order $|D|$ and every cyclic subgroup of $S$ with order  $4$ (if $S$ is non-abelian and $|D|=p=2$) is a strong $q$-$CAP$-subgroup for any prime divisor $q$ of $|G|$. Then $\mathcal{F}_S (G)$  is supersolvable.
\end{corollary}

\begin{proof} 
Let $H$ be an arbitrary subgroup of $G$ such that $H$ is a strong $q$-$CAP$-subgroup for  any prime divisor $q$ of $|G|$. Then for any chief factor $U/V$ of a subgroup $T$ of $G$ containing $H$, it follows that there exists a prime divisor $r$ of $|T|$ such that $r \, |\, |U/V|$. Then it follows from the fact  $H$ is a strong $r$-$CAP$-subgroup of $G$ that $H$ covers or avoids $U/V$. By the choice of $U/V$, we conclude that $H$ is a $CAP$-subgroup of $T$, and so we obtain from the choice of $T$ that $H$ is a strong $CAP$-subgroup of $G$. Hence we conclude that there exists a subgroup $D$ of $S$ with order $1 < |D| < |S|$ such that every subgroup of $S$ with order $|D|$ and every cyclic subgroup of $S$ with order  $4$ (if $S$ is non-abelian and $|D|=p=2$) are strong $CAP$-subgroups of $G$. Applying Theorem \ref{characterization for FSG strong CAP}, we get that the fusion system $\mathcal{F}_S (G)$  is supersolvable.
\end{proof}
As a direct application of the theorem above, we obtain the following characterizations for the structure of finite groups under the assumption that every subgroup of $S$ with order $1 <|D| <S$ for some subgroup $D$ of $S$ or $4$ (if $S$ is non-abelian and $|D| = 2 = p$) is a $CAP$-subgroup of $G$ or a strong $q$-$CAP$-subgroup of $G$ for any prime divisor $q$ of $|G|$. 
\begin{corollary}
Let $G$ be a finite group and $S$ a Sylow $p$-subgroup of $G$, where $p$ is a  prime divisor of $|G|$ such that $(p-1, |G|)=1$. Suppose that there exists a subgroup $D$ of $S$ with order $1 <|D| < |S|$ such that every subgroup of $S$ with order $|D|$ and every cyclic subgroup of $S$ with order  $4$ (if $S$ is non-abelian and $|D|=p=2$) is a strong $CAP$-subgroup of $G$, then $G$ is $p$-nilpotent.
\end{corollary}
\begin{proof}
It follows from Theorem \ref{characterization for FSG strong CAP} that $\mathcal{F}_S (G)$  is supersolvable. Thus we conclude from {{\cite[Theorem 1.9]{SN}}} that $G$ is $p$-nilpotent, as desired.
\end{proof}
\begin{corollary}
Let $G$ be a finite group and $S$ a Sylow $p$-subgroup of $G$, where $p$ is a  prime divisor of $|G|$ such that $(p-1, |G|)=1$. Suppose that there exists a subgroup $D$ of $S$ with order $1 < |D| < |S|$ such that every subgroup of $S$ with order $|D|$ and every cyclic subgroup of $S$ with order  $4$ (if $S$ is non-abelian and $|D|=p=2$) are strong $q$-$CAP$-subgroups of $G$ for any prime divisor $q$ of $|G|$. Then $G$ is $p$-nilpotent.
\end{corollary}
\begin{proof}
It follows from Corollary \ref{corollary FSG strong q-CAP} that $\mathcal{F}_S (G)$  is supersolvable. Therefore we obtain from {{\cite[Theorem 1.9]{SN}}} that $G$ is $p$-nilpotent, as required.
\end{proof}
 Similar to the proof of Theorem \ref{characterization for FSG strong p-CAP} and  Theorem \ref{characterization for FSG strong CAP}, we  are now going to  prove Theorem \ref{main theorem} using the method of fusion systems.
\begin{proof}[Proof of Theorem \ref{main theorem}]
Suppose that the theorem fails, and let $G$ be a counterexample of  minimal order. Denote the fusion system $\mathcal{F}_S (G)$ by $\mathcal{F}$.
\begin{itemize}
\item[\textbf{Step 1.}] Let $H$ be a proper subgroup of $G$ such that $S \cap H \in {\rm Syl}_p (H)$ and $|S \cap H| \geq p|D|$. Then  the fusion system $\mathcal{F}_{S \cap H} (H)$ must be supersolvable.
\end{itemize}

It follows from our hypothesis that every cyclic subgroup $T$ of $S \cap H$ with order $|D|$   or $4$ (If $S \cap H$ is non-abelian and $|S \cap H| /2 >|D| = p = 2$) is a strong $CAP$-subgroup of $G$, and ${\rm exp} (S \cap H) \leq 2 $ if $p=2$. Then every cyclic subgroup $T$ of $S \cap H$ with order $|D|$   or $4$ (if $S \cap H$ is non-abelian and $|S \cap H| /2 >|D| = p = 2$) is a strong $CAP$-subgroup of $H$ by Lemma \ref{inheritance} (4), and ${\rm exp} (S \cap H) \leq 2 $ if $p=2$. Hence  $H$ satisfies the hypothesis of the theorem and it follows from the minimal choice of $G$ that $\mathcal{F}_{S \cap H} (H)$ is supersolvable.
\begin{itemize}
\item[\textbf{Step 2.}] Let $Q \in \mathcal{E}_{\mathcal{F}} ^{*}$, then $|Q| \geq p|D|$. If moreover that $Q \not\unlhd G$, then $N_{\mathcal{F}} (Q)$ is supersolvable.
\end{itemize}

If there exists a subgroup $Q \in \mathcal{E}_{\mathcal{F}} ^{*}$ such that $|Q| < p|D|$, then there exists a subgroup $R$ of $S$ such that $|R| = |D| $  and $Q <R$. It follows immediately that $R \leq C_S (Q)$. Since $Q < R \leq S$, we conclude from the fact $Q$ is a member of $\mathcal{E}_{\mathcal{F}} ^{*}$ that $Q$ is $\mathcal{F}$-essential. Thus we conclude from the concept of $\mathcal{E}_{\mathcal{F}} ^{*}$ that $Q$ is $\mathcal{F}$-centric. Therefore $R \leq C_S (Q) =Z(Q) \leq Q$, a contradiction. Hence we get that $|Q| \geq p|D|$. Now suppose that $Q$ is not normal in $G$. Then $N_G (Q)$ is a proper subgroup of $G$. Since $Q \in \mathcal{E}_{\mathcal{F}} ^{*}$, $Q$ is fully $\mathcal{F}$-normalized or $Q=S$. Clearly $S$ is fully $\mathcal{F}$-normalized, hence $Q$ is always fully $\mathcal{F}$-normalized. By the argument below {{\cite[Definition 2.4]{AK}}}, $S \cap N_G (Q) = N_S (Q) \in {\rm Syl}_p (N_G(Q))$. Since $|N_S (Q)| \geq |Q| \geq p|D|$, it follows that $N_G (Q)$ satisfies the assumption of Step 1, so that $\mathcal{F}_{N_S (Q)} (N_G(Q))=N_{\mathcal{F}} (Q)$ is supersolvable.
\begin{itemize}
\item[\textbf{Step 3.}] $|O_p (G)| \geq p|D|$, and moreover $O_p (G) \leq Z_{\mathfrak{U}} (G)$.
\end{itemize}

Suppose that there does not exist a subgroup $Q \in \mathcal{E}_{\mathcal{F}} ^{*}$ such that $Q \unlhd G$. Then for each $Q \in  \mathcal{E}_{\mathcal{F}} ^{*}$, the fusion system $N_{\mathcal{F}} (Q)$ is supersolvable by Step 2. Applying Lemma \ref{fusion system 1}, $\mathcal{F}$ is supersolvable, a contradiction. Hence there exists a subgroup $Q \in \mathcal{E}_{\mathcal{F}} ^{*}$ such that $Q \unlhd G$. Therefore we obtain from Step 2 that $|O_p (G)| \geq |Q| \geq p|D|$. Now we claim that $O_p (G) \leq Z_{\mathfrak{U}} (G)$.  It follows from $|O_p (G)| \geq p|D|$ that every subgroup $T$ of $O_p (G)$ of order $|D|$ and every cyclic subgroup $T$ of $O_p (G)$ of order $4$ (if  $O_p (G)$ is non-abelian, $|O_p (G)| / 2 > |D| = p =2$) is a strong $p$-$CAP$-subgroup of $G$, and ${\rm exp} (O_p (G)) \leq 2$ if $p=2$. Applying Theorem \ref{main result for characterization of hypercentre}, we obtain that $O_p (G) \leq Z_{\mathfrak{U}} (G)$, as required.
\begin{itemize}
\item[\textbf{Step 4.}] Final contradiction.
\end{itemize}

Assume that $H$ is a proper subgroup of $G$ such that $O_p (G) < S \cap H$ and $S \cap H \in {\rm Syl}_p (H)$. By Step 1 and Step 3, $|S \cap H| >|O_p (G)| \geq p|D|$ and so $\mathcal{F}_{S \cap H} (H)$ is supersolvable. Since $O_p (G) \leq Z_{\mathfrak{U}} (G)$ by Step 3, it yields from Lemma \ref{fusion system 2} that $\mathcal{F}_S (G)$ is supersolvable, a contradiction. This final contradiction completes the proof.
\end{proof}
\begin{corollary}
Let $G$ be a finite group and $S$ a Sylow $p$-subgroup of $G$, where $p$ is a  prime divisor of $|G|$ such that $(p-1, |G|)=1$. Suppose that there exists a subgroup $D$ of $S$ with order $1 < |D| < |S|$ such that one of the following holds:
\begin{itemize}
\item[(1)] $p$ is odd, and every subgroup of $S$ with order $|D|$ is a strong $p$-$CAP$-subgroup of $G$;
\item[(2)] $p = 2$, every subgroup of $S$ with order $|D|$ is a strong $2$-$CAP$-subgroup of $G$, ${\rm exp} (S) \leq 2$, and  every cyclic subgroup of $S$ with order  $4$ is a strong $2$-$CAP$-subgroup of $G$ if $S$ is non-abelian and $|S| /2>|D|=p=2$.
\end{itemize}   
Then $G$ is $p$-nilpotent.
\end{corollary}
\begin{proof}
It follows from Theorem \ref{main theorem} that $\mathcal{F}_S (G)$  is supersolvable. Hence we conclude from {{\cite[Theorem 1.9]{SN}}} that $G$ is $p$-nilpotent, as needed.
\end{proof}
\section{Acknowledgement}
The authors are very grateful to the referee for providing detailed reports, correcting the grammatical errors, and for improving the results presented in the paper considerably.

\end{document}